\documentclass[11pt,reqno]{amsart}

\usepackage[utf8]{inputenc}
\usepackage{lmodern}
\usepackage{amsmath,amsfonts,amssymb,amsthm,enumerate,url}
\usepackage{microtype}
\usepackage[noadjust]{cite}
\usepackage{amssymb,bbm,mathrsfs}

\numberwithin{equation}{section}
\usepackage[colorlinks=true, pdfstartview=FitV, pagebackref=false]{hyperref}
\usepackage{xcolor}
\usepackage{tikz}
\usetikzlibrary{trees,arrows,positioning,automata,shadows,fit,shapes}

\newtheorem{maintheorem}{Theorem}

\newtheorem{theorem}{Theorem}[section]
\newtheorem*{theorem*}{Theorem}
\newtheorem{lemma}[theorem]{Lemma}
\newtheorem{claim}[theorem]{Claim}
\newtheorem{proposition}[theorem]{Proposition}

\newtheorem{corollary}[theorem]{Corollary}

\theoremstyle{definition}{

\newtheorem{definition}[theorem]{Definition}
\newtheorem*{definition*}{Definition}

\newtheorem*{question*}{Question}
\newtheorem*{example*}{Example}
\newtheorem*{examples*}{Examples}

\newtheorem*{remark*}{Remark}

}

\renewcommand{\P}{\mathbb{P}}
\newcommand{\E}{\mathbb{E}}
\DeclareMathOperator{\bP}{\mathbf{P}}
\DeclareMathOperator{\bE}{\mathbf{E}}

\def\Psub_#1{\P_{\! #1}}
\def\Pleft#1{\P\mkern-1mu\left(#1\right)}

\def\Pbig#1{\P\mkern-.5mu\bigl(#1\bigr)}
\def\Psubbig_#1#2{\Psub_{#1}\mkern-1.5mu\bigl(#2\bigr)}

\def\Psubbigg_#1#2{\Psub_{#1}\mkern-1.5mu\biggl(#2\biggr)}
\def\PBig#1{\P\mkern-.5mu\Bigl(#1\Bigr)}

\renewcommand{\P}{\mathbb{P}}
\newcommand{\cA}{{\mathcal A}}

\newcommand{\cC}{{\mathcal C}}

\newcommand{\cE}{{\mathcal E}}
\newcommand{\cF}{{\mathcal F}}

\newcommand{\cS}{{\mathcal S}}

\newcommand{\bR}{{\mathbf R}}
\newcommand{\tx}{\mathtt{tx}}

\newcommand{\red}{\textsc{red}}

\newcommand{\nice}{\textsc{nice}}
\newcommand{\exit}{\textsc{exit}}

\newcommand{\Bin}{\operatorname{Bin}}

\DeclareMathOperator{\dist}{dist}
\DeclareMathOperator{\cov}{Cov}
\renewcommand{\epsilon}{\varepsilon}

\newcommand{\given}{\;\big|\;}
\newcommand{\ind}{{\mathbbm{1}}}
\newcommand{\tmix}{t_\textsc{mix}}
\newcommand{\gap}{\ensuremath{\mathtt{gap}}}

\newcommand{\tv}{{\textsc{tv}}}
\newcommand{\dimH}{{\mathbf{d}}}

\newcommand{\bh}{{\mathbf{h}}}

\usepackage{color}
 \definecolor{refkey}{gray}{.5}
 \definecolor{labelkey}{gray}{.5}
\definecolor{light}{gray}{.9}

\begin{document}

\title{Comparing mixing times on sparse random graphs}
\author[A. Ben-Hamou]{Anna Ben-Hamou}
 \address{A. Ben-Hamou\hfill\break
Sorbonne Universit\'e, LPSM\\
4, place Jussieu\\ 75005 Paris, France.}
\email{anna.ben-hamou@upmc.fr}

\author[E. Lubetzky]{Eyal Lubetzky}
\address{E.\ Lubetzky\hfill\break
Courant Institute\\ New York University\\
251 Mercer Street\\ New York, NY 10012, USA.}
\email{eyal@courant.nyu.edu}

\author[Y. Peres]{Yuval Peres}
\address{Y.\ Peres\hfill\break
Microsoft Research\\ One Microsoft Way\\ Redmond, WA 98052, USA.}
\email{peres@microsoft.com}


\begin{abstract}
It is natural to expect that nonbacktracking random walk will mix faster than simple random walks, but so far this has only been proved in regular graphs. To analyze typical irregular graphs, let $G$ be a random graph on $n$ vertices with minimum degree 3 and a degree distribution that has exponential tails.  We determine the precise worst-case mixing time for simple random walk on $G$, and show that, with high probability,
it exhibits cutoff at time $\bh^{-1} \log n$, where $\bh$ is the asymptotic entropy for simple random walk on a Galton--Watson tree that  approximates $G$ locally. (Previously this was only known for typical starting points.) Furthermore, we show this asymptotic mixing time is strictly larger than the mixing time of nonbacktracking walk, via a delicate comparison of entropies on the Galton--Watson tree.
\end{abstract}
\maketitle
\section{Introduction}

We study the mixing time of simple random walk (SRW) vs.\ the nonbacktracking random walk (NBRW) on sparse random graphs. It is natural to expect, as highlighted in~\cite{alon2007non} for the case of regular expander graphs, that forbidding the walk to backtrack (traverse an edge twice in a row) would allow the walk to mix faster.
It was thereafter shown in~\cite{LS10} that, on a typical random $d$-regular graph on $n$ vertices, both walks exhibit the \emph{cutoff phenomenon}---a sharp transition in the total-variation distance from equilibrium, dropping from near its maximum to near 0 over a negligible time period referred to as the cutoff window: SRW mixes at time $\frac{d}{d-2}\log_{d-1}n + O(\sqrt{\log n})$, as conjectured by Durrett~\cite{Durrett-RGD}, whereas the NBRW mixes faster, having cutoff already at time $\log_{d-1}n + O(1)$. Our goal here is to obtain analogous results for the irregular case.

The (worst-case) total-variation distance of a Markov chain with transition kernel $P$ and state space $\Omega$ from its stationary distribution $\pi$ is defined as
\[
d_\tv(t)= \max_{x_0\in\Omega}\left\|P^t(x_0,\cdot)-\pi\right\|_\tv = \max_{x_0\in\Omega}\sup_{A\subset\Omega}\left|P^t(x_0,A)-\pi(A)\right|\,,
\]
and, for $0<\epsilon<1$, its corresponding mixing time to within distance $\epsilon$ is given by
\[
\tmix(\epsilon) = \min\left\{ t : d_\tv(t)\leq\epsilon\right\}\,.
\]
(When addressing a prescribed initial state $x_0\in\Omega$ rather than the worst one, this definition is replaced by $\tmix^{(x_0)}(\epsilon)=\min\{ t : d_\tv^{(x_0)}(t)\leq \epsilon\}$ where $d_\tv^{(x_0)}=\|P^t(x_0,\cdot)-\pi\|_\tv$.)
The notion of cutoff, due to Aldous, Diaconis and Shahshahani~\cite{Aldous,AD,DiSh}, captures the dependence of $\tmix$ on the parameter $\epsilon$: a sequence of chains is said to exhibit cutoff if
$\lim_{|\Omega|\to\infty}[\tmix(\epsilon) / \tmix(\epsilon')]= 1$ for every fixed $0<\epsilon,\epsilon'<1$.


Roughly put, the locally-treelike geometry of sparse random graphs makes the mixing behavior of random walk resemble that on trees. In the regular case, a walk on a regular tree, conditioned to be at a given distance from the root, is uniform on that level, and the analogous mixing time on a random $d$-regular graph coincides with the time $t$ at which this distance is about $\log_{d-1} n$ (so as to contain almost all vertices in its range). This corresponds to $t = \log_{d-1} n +O(1)$ for the NBRW, and a slowdown of the SRW by a factor of $d/(d-2)$ due to the reduced speed of random walk on a tree (with a coarser $O(\sqrt{\log n})$-window due to the normal fluctuations of its height), as in the results of~\cite{LS10}.

In the irregular case, however, the walk can have vastly different probabilities of traversing different paths, and mixing occurs when almost all vertices are---not only reachable by paths---but reachable by ones that are sufficiently probable. It might then be the case that the NBRW, allbeit faster to reach the leaves of a tree,  could potentially be ``trapped'' in a set of lower-probability paths, compared with the  SRW that retains a decent probability of backtracking and exploring more favorable parts of the tree.
It turns out that the effect of such traps is not strong enough to compensate for the slowdown of the SRW, and, even in the irregular case, the NBRW mixes faster.




Let $G=(V,E)$ be a random graph on $n$ vertices with vertex set $V$ and degree distribution $(p_k)_{k\geq 1}$; that is, the degree $D_x$ of each vertex $x\in V$ is independently sampled with $\P(D_x=k)=p_k$, conditioned on  $\sum_x D_x$ being even, and $G$ is thereafter generated by the configuration model. Let $Z$ be a random variable with distribution
\begin{eqnarray*}
\P(Z=k-1)&=&\frac{kp_{k}}{\sum_{\ell\geq 1}\ell p_\ell}\, ,
\end{eqnarray*}
consider a rooted Galton--Watson tree $(T,\rho)$ with offspring variable $Z$, and let $(X_t)$ and $(Y_t)$ be  SRW and NBRW, respectively, on $T$ started at $\rho$. The mixing times on $G$ can be expressed in terms of the asymptotic entropy of these walks on $T$ as follows. Let
\begin{align}\label{eq-hX-def}
\bh_X &\stackrel{\text{a.s.}}= \lim_{t\to\infty}\frac1t H\Big(\P_\rho(X_t\in\cdot \mid T)\Big)\quad\mbox{ and }\quad
\bh_Y\stackrel{\text{a.s.}}= \lim_{t\to\infty}\frac1t H\Big(\P_\rho(Y_t\in\cdot \mid T)\Big)\, ,
\end{align}
where the entropy $H(\mu_T)$ of a probability measure $\mu_T$ on the vertices of $T$ is given by
\begin{align*}
H\left(\mu_T\right)=-\sum_{x\in T} \mu_T(x)\log \mu_T(x)\,.
\end{align*}
It was shown in~\cite{BLPS} that, when the initial vertex $v_1$ is \emph{fixed} (independently of the graph) and the degree distribution $(p_k)_{k\geq 1}$ satisfies suitable moment assumptions, with high probability (w.h.p.) the SRW has cutoff at time $\bh_X^{-1}\log n$ whereas the NBRW has cutoff at time $\bh_Y^{-1}\log n$. Comparing these two mixing times was left open.

Comparing this with the regular case, observe that, for the NBRW, $\bh_Y = \E \log Z$, which satisfies $\bh_Y < \log\E Z$ whenever $Z$ is not a constant by Jensen's inequality. Hence, the NBRW mixes well after the time at which its range covers most vertices (local neighborhoods are approximately Galton--Watson trees with offspring variable $Z$), unlike the regular setting. The same phenomenon occurs for SRW: denoting by $\nu$ the limiting speed of SRW on $T$, then $\bh_X/\nu < \log \E Z$ whenever $Z$ is not a constant (the ``dimension drop'' of harmonic measure), as shown in~\cite{lyons1995ergodic}.

From a \emph{worst-case} initial vertex, it was showed independently in~\cite{BS,BLPS} that w.h.p.\ the NBRW also exhibits cutoff at $\bh_Y^{-1}\log n$ under similar moment assumptions (in~\cite{BS} the Gaussian tail of the distance profile within the cutoff window was further established). Here we extend the arguments of~\cite{BLPS} to provide the analogous cutoff result for the SRW from a worst-case starting point, as well as compare these cutoff locations.



\begin{maintheorem}\label{mainthm:main}
Let $G$ be a random graph with degree distribution $(p_k)_{k=1}^\infty$, such that
\begin{align}\label{eq-Z-min-degree-hypo}
p_1=p_2=0\,,\qquad \sum k^2 p_k < \infty\,,
\end{align}
and for some fixed $\delta>0$, the random variable $Z$ given by $\P(Z=k-1) \propto  k p_k $ satisfies
\begin{align}
\label{eq-Z-max-degree-hypo}
 \P(Z > \Delta_n) = o(1/n) \quad\mbox{ for }\quad \Delta_n := \exp\big[(\log n)^{1/2-\delta}\big]\,.
 \end{align}
Then w.h.p., \emph{SRW} from a worst-case vertex has cutoff at $\bh_X^{-1} \log n$ with window $\sqrt{\log n}$,
with $\bh_X$ as defined in~\eqref{eq-hX-def} for a Galton--Watson tree $T$ with offspring variable $Z$.
 Moreover, $\bh_X < \bh_Y$ and the \emph{NBRW} mixes faster than \emph{SRW}.
\end{maintheorem}
Condition~\eqref{eq-Z-max-degree-hypo} is weaker than requiring $Z$ to have an exponential tail. 
Controlling the maximum degree is crucial for the local approximation of $G$ by the Galton--Watson tree $T$ to be valid; the specific bound $\exp[(\log n)^{1/2-\delta}]$ was chosen in light of the fact that one can control the probability distribution of SRW on $T$ at time  $c\log n$ up to multiplicative factors of $\exp[O(\sqrt{\log n})]$.
Further note that the assumption $p_1=p_2=0$ in~\eqref{eq-Z-min-degree-hypo} is needed, otherwise SRW would typically mix in time $c\log^2 n$ with no cutoff.

\medskip

The two statements of Theorem~\ref{mainthm:main} are proved separately. In Section~\ref{sec:worst-case-cutoff} we prove Proposition~\ref{prop:worst-case}, which extends the result of~\cite{BLPS} to the worst-case starting point.
In Section~\ref{sec:entropy} we establish Proposition~\ref{prop:nbrw-faster}, which shows that $\bh_X< \bh_Y$ when the degrees are at least $3$. Together, this implies that, as in the regular case, the backtracking moves of the SRW on sparse irregular random graphs delay its mixing compared to the NBRW.

\begin{proposition}\label{prop:worst-case}
Under the assumptions~\eqref{eq-Z-min-degree-hypo}--\eqref{eq-Z-max-degree-hypo}, w.h.p.\ the worst-case mixing time of the 
\emph{SRW} on $G$ exhibits cutoff at time $\bh_X^{-1} \log n$ with window $O(\sqrt{\log n})$.
\end{proposition}

\begin{proposition}\label{prop:nbrw-faster}
Let $Z$ be a random variable taking nonnegative integer values such that $\E Z < \infty$ and $Z\geq 2$ a.s. Then
$\bh_X < \bh_Y$ on the GW-tree with offspring variable $Z$.
\end{proposition}


\section{Simple random walk from the worst starting point}\label{sec:worst-case-cutoff}
In this section we prove Proposition~\ref{prop:worst-case}, establishing cutoff and its location for SRW on a sparse random graph with degree distribution $(p_k)_{k=0}^\infty$. Recalling that the minimum degree is $3$ by assumption ($p_1=p_2=0$), let $\Delta=\Delta(n)$ denote the maximum degree in $G$, which we may assume is at most $\exp[(\log n)^{1/2-\delta}]$ for some $\delta>0$ fixed.
The lower bound on the cutoff window from a worst-case starting point follows immediately from~\cite[Theorem~2]{BLPS} (where it was established for a uniformly chosen initial vertex, and in particular carries to the worst one), and it remains to show a matching upper bound.

The first step in the proof is to reduce the analysis of worst-case starting points, as was done in~\cite{LS10}, to vertices whose neighborhood up to distance $c \log\log n$, for an appropriate constant $c>0$, is a tree. For $x\in V$ and $k>0$, let $B_k(x)=\{ y : \dist(x,y) \leq k\}$ be the $k$-radius neighborhood of $x$, and let $\partial B_k(x) = \{y:\dist(x,y)=k\}$.
Further define the \emph{tree excess} of $B_k(x)$, denoted $\tx(B_k(x))$, to be the maximum number of edges that can be deleted from the induced subgraph on $B_k(x)$ while keeping it connected.
\begin{definition}\label{def:K-root}
Let $K>0$. A vertex $x\in V$ is said to be a $K$-root if $\tx(B_K(x))=0$.
\end{definition}
The next two straightforward lemmas follow~\cite[Lemmas~2.1 and~3.2]{LS10}.
\begin{lemma}\label{lem:tree-excess}
Let $K=O(\sqrt{\log n})$. W.h.p., every vertex $x\in V$ satisfies $\tx(B_{5K}(x))\leq 1$.
\end{lemma}
\begin{proof}
Fix $x\in V$, and condition on the degrees $\{D_u : u\in V\}$, noticing $\sum D_u \geq 3 n$ since $p_1=p_2=0$.
 We generate the ball of radius $5K$ around $x$ sequentially, by the standard breadth-first search process of the configuration model ({\it cf.}, {\it e.g.},~\cite{Bollobas-RG}): identifying each vertex $u$ with $\deg(u)$ ``half-edges,'' denoted $(u,*)$, we start by inserting all half-edges of $x$ to a first-in-first-out queue. Upon extracting a half-edge $(u,*)$ from the queue it is matched to a uniformly chosen unmatched half-edge $(v,*)$ in $G$, to form the edge $(x,y)$, and all other half-edges of the vertex $v$ are inserted to the queue (if not already there).

If $M$ is the total number of half-edges encountered until first observing a vertex at distance $5K+1$ from $x$ (whence $B_{5K}(x)$ is fully exposed), then  clearly we will perform at most $M$ steps of extracting a half-edge from the queue and matching it, and in each such step, the number of possible matches in the queue (adding a cycle in $G$ and increasing $\tx(B_{5K}(x))$) is at most $M$ vs.\ at least $3n-M$ half-edges outside the queue.
Since $M \leq \Delta^{5K+1} $, it follows that $ \tx(B_{5K}(x)) $ is stochastically dominated by a  binomial random variable $\Bin(\Delta^{5K+1}, \Delta^{5K+1}/(3n-\Delta^{5K+1}))$, and noting that $\Delta^{5K+1}  \leq  \exp\left[ O((\log n)^{1-\delta})\right] = n^{o(1)}$ by our assumptions on $K$ and $\Delta$,  we find that
\[
\PBig{\tx(B_{5K}(x))\geq 2} \leq \binom{\Delta^{5K+1}}2 \bigg( \frac{\Delta^{5K+1}}{(3-o(1))n}\bigg)^2
= n^{-2+o(1)}\, .
\]
A union bound over $x$ concludes the proof.
\end{proof}

\begin{lemma}\label{lem:reach-root}
Let $K>0$, and let $x$ be a vertex in a  graph $H$ with minimum degree 3, such that $\tx(B_{5K}(x))\leq 1$. Then \emph{SRW} of length $4K$ started at $x$ ends at a $K$-root with probability at least $1-e^{-K/128}$.
\end{lemma}
\begin{proof}
 Let $(X_t)$ be SRW on $H$ started at $x$. If $x$ is a $5K$-root, then the claim is trivial. Otherwise, the induced subgraph on $B_{5K}(x)$ has exactly one cycle $\mathcal{C}$, by assumption. We claim that in this situation,
 \begin{equation}
   \label{eq-dist-Xt-C}
   \Pbig{ \dist(X_{4K},\cC)  < K } \leq \exp(- K/128)\,.
 \end{equation}
 Let $\rho_t=\dist(X_t,\mathcal{C})$. By the minimum degree assumption, $\P(\rho_{t+1}=0 \mid \rho_t=0\,,\, X_t) \leq \frac23$, whereas on the event $\rho_t > 0$ one has
$|\rho_{t+1}-\rho_t|=1$ and $\P(\rho_{t+1} < \rho_t \mid \rho_t>0\,,\,X_t) \leq \frac13$. Combining both situations,
$\E[\rho_{t+1}-\rho_t\given X_t]\geq  \frac{1}{3}$, so $M_t = \rho_t -\frac13 t$ is a submartingale with $|M_{t+1}-M_t| \leq \frac43$.
By the Hoeffding--Azuma inequality,
\begin{eqnarray*}
\Pleft{\rho_{4K} \leq \rho_0 +K} = \Pleft { M_{4K} \leq M_0 - \frac{K}3 }\leq \exp\biggl(-\frac{(K/3)^2}{2(\frac43)^2 \cdot 4K}\biggr) = e^{-K/128}\,,
\end{eqnarray*}
establishing~\eqref{eq-dist-Xt-C} and concluding the proof.
\end{proof}
Observe that Lemma~\ref{lem:tree-excess} gives, w.h.p., the hypothesis of Lemma~\ref{lem:reach-root} simultaneously for all $x\in V$ provided that $K=O(\sqrt{\log n})$, implying then that
\begin{equation}
  \label{eq-K-root-reduction}
  d_\tv(t + 4K)  \leq \max\left\{ d_\tv^{(x_0)}(t) \,:\; x_0\mbox{ is a $K$-root}\right\}+e^{-K/128}
\end{equation}
holds w.h.p. By choosing a suitably large $K$, this will allow us to reduce the problem to initial vertices that are $K$-roots.
To bound the mixing time from a $K$-root $x_0$, let
\[ \tau_K = \tau_K(x_0) := \inf\left\{ t>0 \,:\; X_t \in \partial B_K(x_0)\right\} \,,\]
where $X_t$ is the SRW on $G$ started at $x_0$, and further define
\[
\Lambda_t =\Lambda_t(x_0,K,\epsilon) := \left\{x\in \partial B_K(x_0)\,:\; d_\tv^{(x)}(t)>\varepsilon\right\}\, .
\]
The next simple lemma bounds $d_\tv^{(x_0)}$ via the probability that $X_{\tau_K} \in \Lambda_t$ started from $x_0$.
\begin{lemma}\label{lem:TV_decomp}
For every $x_0\in V$ and every $\epsilon,K,s,t>0$,
\[
d_\tv^{(x_0)}(t+s)
\leq \tfrac12 \P_{x_0}^G\left(\tau_K>s\right)+ \P_{x_0}^G\left(X_{\tau_K} \in \Lambda_t\right) +\epsilon \,.
\]
\end{lemma}
\begin{proof}
Write $S_K = \partial B_K(x_0)$ for brevity.
By the triangle inequality,
\begin{align*}
 \left\| \P_{x_0}^G\left(X_{t+s}\in\cdot\right) -\pi\right\|_1
\leq  \P_{x_0}^G\left(\tau_K>s\right)+ \left\| \P_{x_0}^G\left(X_{t+s}\in\cdot,\, \tau_K\leq s\right)-\pi\right\|_1\, .
\end{align*}
Using the strong Markov property, $\P_{x_0}^G\left(X_{t+s}\in\cdot,\, \tau_K\leq s\right)$ is equal to
\begin{align*}
& \sum_{\ell=1}^s\sum_{z\in S_K} \P_{x_0}^G\left(\tau_K=\ell,X_{\tau_K}=z\right) \P_{x_0}^G\left(X_{t+s}\in\cdot\given \tau_K=\ell,X_{\tau_K}=z \right)\\
&= \sum_{\ell=1}^s\sum_{z\in S_K} \P_{x_0}^G\left(\tau_K=\ell,X_{\tau_K}=z\right)
\P_z^G\left(X_{t+s-\ell}\in\cdot\right)\, .
\end{align*}
Combining these two statements, $d_\tv^{(x_0)}(t+s) = 
 \tfrac12 \left\| \P_{x_0}^G\left(X_{t+s}\in\cdot\right) -\pi\right\|_1$ satisfies
\begin{align*}
d_\tv^{(x_0)}(t+s)&\leq \tfrac12\P_{x_0}^G\left(\tau_K>s\right)+\sum_{\ell=1}^s\sum_{z\in S_K} \P_{x_0}^G\left(\tau_K=\ell,\,X_{\tau_K}=z\right)d_\tv^{(z)}(t+s-\ell)\\
&\leq  \tfrac12\P_{x_0}^G\left(\tau_K>s\right)+\sum_{z\in S_K} \P_{x_0}^G\left(X_{\tau_K}=z\right)d_\tv^{(z)}(t)\, ,
\end{align*}
where we used that $t\mapsto d_\tv^{(z)}(t)$ is non-decreasing. The proof is concluded by breaking the summation over $z$ to $S_K \cap \Lambda_t$ and $ S_K \cap \Lambda_t^c$ and using the bound $d_\tv^{(z)}(t)\leq \epsilon$ in the latter case (by definition of $\Lambda_t$) and
$d_\tv^{(z)}(t)\leq 1$ in the former.
\end{proof}
For every $K$-root $x_0\in V$, one has
\[
\P_{x_0}^G\left(\tau_K>4K\right) \leq e^{-K /128}
\]
by the exact same proof of~\eqref{eq-dist-Xt-C} (the root vertex $x_0$ plays the role of $\cC$ in that argument).
Together with~\eqref{eq-K-root-reduction} and Lemma~\ref{lem:TV_decomp}, this implies the following.
\begin{corollary}\label{cor:dtv-reduction}
For all $\epsilon,t>0$ (that may depend on $n$) and $K=O(\sqrt{\log n})$, w.h.p.
\begin{equation}
  \label{eq-Bt-reduction}
  d_\tv(t + 8K)  \leq \max\left\{ \mu_{x_0}(\Lambda_t) \;:\; x_0\mbox{ is a $K$-root}\right\}+\epsilon + \tfrac32e^{-K/128}\,,
\end{equation}
where $\mu_{x_0}=\P_{x_0}^G(X_{\tau_K} \in \cdot)$ is the hitting measure of the \emph{SRW} from $x_0$ on $\partial B_K(x_0)$.
\end{corollary}
Combining this corollary with the following theorem will conclude the proof.
\begin{theorem}\label{thm:mu(Bt)}
For every $\epsilon>0$ there exists some  $\gamma_\star>0$ such that, if
\begin{equation}
\label{eq:t1-K-def} t_1 = \bigl\lceil \bh_X^{-1} \log n + \gamma_\star\sqrt{\log n}\bigr\rceil\qquad\mbox{ and }\qquad K = \left\lceil \gamma_\star \log \log n \right\rceil\,,
\end{equation}
then  for every tree $T_0$ such that $\P(B_K(x_0)=T_0)>0$ and sufficiently large $n$,
  \[ \PBig{\mu_{x_0}(\Lambda_{t_1}) > \epsilon \given B_K(x_0) = T_0} < n^{-2}\,.\]
\end{theorem}
Indeed, modulo Theorem~\ref{thm:mu(Bt)}, for every fixed $\epsilon>0$, we may take $t_1$ and $K$ as in~\eqref{eq:t1-K-def}, and deduce (via a union bound) that, w.h.p., every $K$-root vertex $x_0$ satisfies $\mu_{x_0}(\Lambda_{t_1}) \leq \epsilon$. Therefore, $d_\tv(t_1+8K) \leq 2\epsilon + o(1)$ w.h.p.\ by~\eqref{eq-Bt-reduction}, yielding the desired upper bound in Proposition~\ref{prop:worst-case}.

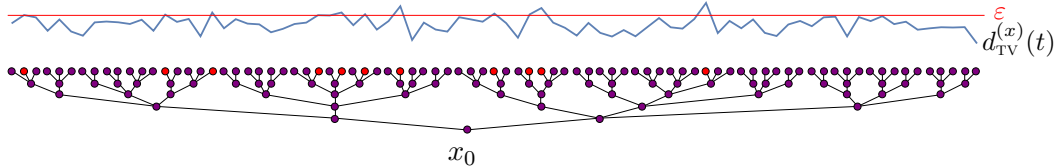
\begin{figure}
\label{fig:BK(x0)}

\begin{tikzpicture}

 \node (fig1) at (0,0) {
	\includegraphics[width=0.9\textwidth]{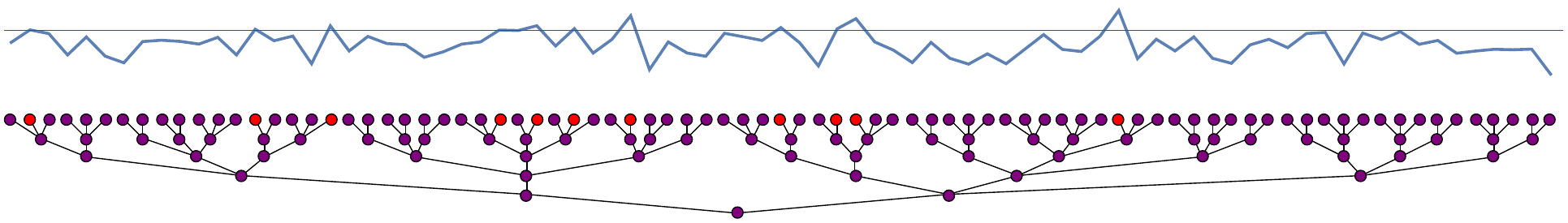}};
	
  \node at (-0.45,-1.2) { $x_0$ };
  \node[red,font=\small] at (6.7,0.7) { $\epsilon$ };
  \node[blue!0!black,font=\small] at (6.95,0.35) { $d_{\tv}^{(x)}(t)$ };

\end{tikzpicture}
\caption{The set $\Lambda_t \subset \partial B_K(x_0)$ marked in red, above which are  the corresponding values of $d_{\tv}^{(x)}(t)$ for $x\in\partial B_K(x_0)$, with a threshold at $\epsilon$.}
\end{figure}
\begin{proof}[\textbf{\emph{Proof of Theorem~\ref{thm:mu(Bt)}}}]

Throughout the proof of this theorem, let $\bP^{T_0}$ denote the conditional probability $\P(\cdot\mid B_K(x_0)=T_0)$, and similarly let $\bE^{T_0}$ denote the analogous conditional expectation.

For a given vertex $x\in V$, consider a breadth-first-search exploration of $G$ from $x$, as described in the proof of Lemma~\ref{lem:tree-excess}, with the following addition: upon matching a half-edge $(u,*)$ from the queue to some (uniformly chosen) unmatched half-edge $(v,*)$,  if the vertex $v$ had already been encountered ({\it i.e.}, $(v,*)$ is in the queue), we stop the exploration at $u$ and at $v$, and mark $u$ and $v$ as \exit\ vertices. Also, when $u$ is at distance $L$ from $x$ (for some integer $L$ to specified later), we stop the exploration at this vertex and mark it as an \exit\ vertex. Let $\Gamma_L(x)\subset B_L(x)$ denote the tree, rooted at $x$, that is obtained by running this process until no further exploration is allowed.

We first recall some key facts established in~\cite{BLPS} in the analysis of SRW on sparse irregular random graphs.
For a given graph $G=(V,E)$ on $n$ vertices and vertex $x\in V$,  an integer $t>0$ and constants $\gamma,\epsilon>0$,  define the event
\begin{equation}
   \label{eq:cA-def}
   \cA^G_x(t,\gamma,\epsilon) = \bigg\{ \sum_{y\in V} \left( \P_{x}^G(X_t = y) \wedge n^{-1} e^{\gamma\sqrt{\log n}}\right) \geq 1-\epsilon \bigg\}\,.
 \end{equation}
Note that, if $\gap$ is the spectral-gap of SRW on a graph $G$ (note that $\gap$ is w.h.p.\ uniformly bounded away from $0$, by the well-known fact that, given the assumption $p_1=p_2=0$, the random graph $G$ is w.h.p.\ an expander), then for large enough $n$,
\begin{equation}\label{eq:gap-boost} \cA^G_x(t_0,\gamma,\epsilon) \qquad\Rightarrow\qquad d_{\tv}^{(x)}\left(t_0+s \right) < 2\epsilon \quad\mbox{ for }\quad s = 2\gamma \gap^{-1}\sqrt{\log n}\,.\end{equation}
Indeed, on the event $\cA^G_x$, the sub-measure $\nu(y) := P^{t_0}(x,y) \wedge n^{-1}\exp(\gamma\sqrt{\log n})$ satisfies $\nu(G) \geq 1-\epsilon$, and using that $1/(\Delta n) \leq \pi(y)\leq \Delta/n$ for all $y$ (together with the assumption  $\Delta=\exp(o(\sqrt{\log n}))$) we get  $\|\nu/\pi - \nu(G)\|_{L^2(\pi)} \leq \exp((\gamma+o(1))\sqrt{\log n})$, so
\[d_{\tv}^{(x)}(t_0+s) \leq \epsilon + \|\nu P^s/\pi - 1 \|_{L^2(\pi)} \leq \epsilon + \exp(-(\gamma-o(1))\sqrt{\log n})) = \epsilon + o(1)\,.\]
 
It was shown in~\cite{BLPS} that, if $G$ is a random graph as specified in Theorem~\ref{mainthm:main} and $x$ is a uniformly chosen vertex in $G$, then for every $\epsilon>0$ there exist  $\gamma_1,\gamma_2>0$ such that
\begin{equation}\label{eq:t0}
 \PBig{\cA^G_x(t_0,\gamma_1,\epsilon)}\geq 1-\epsilon  \qquad\mbox{ for }\qquad  t_0 :=\left\lceil \bh_X^{-1} \left( \log n - \gamma_2 \sqrt{\log n}\right)\right\rceil \,.	
\end{equation}
(See~\cite[Eq.~(3.19)]{BLPS}, recalling $\bh_X = \nu \dimH$, where $\nu$ and $\dimH$ are respectively the speed and Hausdorff dimension of harmonic measure of SRW on a GW-tree with offspring distribution $Z$, and the definitions in~\cite[Eqs.~(3.1)--(3.3)]{BLPS}: our $\epsilon$ here replaces $\sqrt{5\epsilon}$ from that equation, and $\gamma_1$ and $\gamma_2$ replace $\frac43\gamma$ and $\frac78\gamma$, respectively.)
%
Moreover, by introducing an object referred to as \emph{truncated} GW-trees~\cite[\S3.1]{BLPS}, it was shown there that for all $\varepsilon >0$, there exist $c_0,\gamma_3>0$, depending only on $\varepsilon$ and on the law of $Z$, such that, if
\begin{equation}
  \label{eq:def-L-R}
R := \big\lceil c \log \log n\big\rceil \,, \qquad L =\big\lceil \nu \big( t_0 + \bh_X^{-1} \gamma_3\sqrt{\log n}\big)\big\rceil \,,
\end{equation}
for some fixed $c\geq c_0$,
and  $x\in V$ is  fixed independently of $G$, then for large $n$ there exists a subtree $\Gamma'_L(x)  \subset \Gamma_L(x)$ (the \emph{truncation} of $\Gamma_L(x)$, in which the exploration is also stopped at vertices which the SRW is less likely to visit, with those vertices being marked as \exit\ vertices as well)  with
the following properties:
\begin{enumerate}[(i)]
\item \label{it-first-R-min-degree} The first $R$ levels of $ \Gamma'_L(x)$ and $\Gamma_L(x)$ are equal, and can be coupled w.h.p.\ to a standard GW-tree\footnote{To be precise, the offspring distribution of the root is different --- being given by $(p_k)$ vs.\ all other vertices where it is given by $(q_k)$ for $q_k \propto (k+1)p_{k+1}$ --- but this does not essentially change the proofs.}  with variable $Z$ (so every vertex there has at least $2$ offspring).

\noindent (See the definition of truncation in~\cite[\S3.1]{BLPS} and the coupling in~\cite[\S3.2]{BLPS}.)

\item\label{it-tree-size} The total size of the tree $ \Gamma'_L$ is at most $n \exp(-\frac14\gamma_1\sqrt{\log n})$.

\noindent (See~\cite[Eq.~(3.11)]{BLPS}, where $\ell_1$ equals $L$ from~\eqref{eq:def-L-R} for  $\gamma_3=\frac18 \gamma$ and $\gamma_1,\gamma_2$ as before.)

\item \label{it-coupling}
The SRW $(X_t)_{t=1}^{t_0}$  on $G$ from $x$ does not hit any \exit\ vertex before time $t_0$ with probability $1-O(\epsilon)$ and can be coupled to the SRW on a \emph{truncated} GW-tree with probability $1-O(\epsilon)$. Denote the event of such a successful coupling by $\Pi_{t_0}$. Furthermore,  $\Gamma'_L(x)$ can be constructed while revealing at most $n\exp(-\frac14\gamma_1\sqrt{\log n})$ vertices of $B_L(x)$.

\noindent (See the description of this coupling in~\cite[\S3.2]{BLPS} and the 
event $\Pi_k$ defined there.)
\item \label{it-implying-A}
There exists an event $\cE^G_x \in \sigma(\Gamma'_L(x))$ such that
\begin{equation}\label{eq:cE-def} \cE^G_x \subset \cA^G_x(t_0, \gamma_1,\epsilon)\qquad\mbox{ and }\qquad
    \P(\cE^G_x) \geq 1-\epsilon\,.
\end{equation}    

\noindent (See~\cite[\S3.4.1]{BLPS} and the event $\{ \P_G(\Upsilon) \geq 1-\sqrt{5\epsilon}\}$ there, which is our event $\cE^G_x$.)

\end{enumerate}

We claim that this analysis extends to the case where $x$, instead of being a fixed vertex (chosen independently of the random graph $G$), is taken to be a leaf of $B_K(x_0)$.

Consider the following exploration: initially, $B_K(x_0)$ is fully exposed, and we assume $B_K(x_0)=T_0$ for some tree $T_0$. Now, for $x\in\partial T_0$, we may expose the truncated tree $\Gamma'_L(x)$ just as in~\cite{BLPS}, the only difference being that some part of this tree, namely $T_0$, is assumed to be exposed already at the beginning (no truncation will occur there). We will also consider a variant of this exposure process: given $B_K(x_0)=T_0$, for $x\in\partial T_0$ with ancestor $z\in\partial B_{\lfloor K/2\rfloor}(x_0)$, we sequentially expose the neighborhood of $x$, stopping the exploration at $u$ when we are about to create a cycle ({\it i.e.}, a half-edge $(u,*)$ is matched to a previously encountered vertex) or when $u$ meets the truncation criterion~\cite[\S3.1]{BLPS}, just as for $\Gamma'_L(x)$, but in addition, we also stop the exploration at every $u\in \partial T_0$ which is not a descendant of $z$ (and, just as before, mark every such vertex as an \exit\ vertex). 
Let $\widetilde\Gamma_L(x)$ be the subtree obtained this way. An illustration of this exposure process is depicted in Figure~\ref{fig:exposure}.
\begin{figure}
\label{fig:exposure}

\tikzstyle{level 1}=[sibling angle=120]
\tikzstyle{level 2}=[sibling angle=60]
\tikzstyle{level 3}=[sibling angle=40]
\tikzstyle{level 4}=[sibling angle=20]

\begin{tikzpicture}[scale=0.5,sibling distance=8em,grow cyclic,
  every node/.style = {fill,circle,minimum size=1.7mm,inner sep=0}]
  \node[label={[xshift=0.2cm, yshift=0.01cm]\large{$x_0$}}] { }
    child { node { } 
       child { node[green!70!black,star,minimum size=0.3cm] { } 
           child { node { } 
              child { node { } }
              child { node { } }
              child { node { } }
              }
           child { node { } 
              child { node { } }
              child { node { } }
              }
           }        
       child { node[green!70!black,star,minimum size=0.3cm] { } 
           child { node { } 
              child { node { } }
              child { node { } }
              }
           child { node { } 
              child { node { } }
              child { node { } }
              child { node { } }
              }
           }
       child { node[green!70!black,star,minimum size=0.3cm] { } 
           child { node { } 
              child { node { } }
              child { node { } }
              child { node { } }
              }
           child { node { } 
              child { node { } }
              child { node { } }
              }
           child { node { } 
              child { node { } }
              child { node { } }
              }
           }
        }
    child { node { } 
       child { node[green!70!black,star,minimum size=0.3cm] { } 
           child { node { } 
              child { node { } }
              child { node { } }
              }
           child { node { } 
              child { node { } }
              child { node { } }
              child { node { } }
              }
           child { node { } 
              child { node { } }
              child { node { } }
              }
           }
       child { node[green!70!black,label={[xshift=-0.05cm, yshift=0.1cm]
       {\color{green!60!black}$z$}},star,minimum size=0.3cm] { } 
           child { node { } 
              child { node { } child[red!80!black,dashed]{child[red!80!black,dashed]{child[red!80!black,dashed]{}}} }
              child { node { } child[red!80!black,dashed]{child[red!80!black,dashed]{child[red!80!black,dashed]{}}} }
              child { node { } child[red!80!black,dashed]{child[red!80!black,dashed]{child[red!80!black,dashed]{}}} }
              }
           child { node { } 
              child { node { } child[red!80!black,dashed]{child[red!80!black,dashed]{child[red!80!black,dashed]{}}} }
              child { node[blue!80!black,label={[xshift=-0.3cm,yshift=-0.2cm]
              {\color{blue!80!black}$x$}}] { } child[red!80!black,dashed]{child[red!80!black,dashed]{child[red!80!black,dashed]{}}} }
              }
           }
       }
    child { node { } 
       child { node[green!70!black,star,minimum size=0.3cm] { } 
           child { node { } 
              child { node { } }
              child { node { } }
              }
           child { node { } 
              child { node { } }
              child { node { } }
              }
           child { node { } 
              child { node { } }
              child { node { } }
              child { node { } }
              }
           }
       child { node[green!70!black,star,minimum size=0.3cm] { } 
           child { node { } 
              child { node { } }
              child { node { } }
              child { node { } }
              }
           child { node { } 
              child { node { } }
              child { node { } }
              }
           }
       };
  \draw[dotted,red] (7.,7.2) arc (70:-20:5.7);
  
  \node[white] at (12,5) {\Large\color{red!80!black}$\widetilde\Gamma_L(x)$}; 
    \node[white] at (-2.5,0.3) {\Large\color{black} $T_0$};
\end{tikzpicture}

\caption{The tree $T_0=\partial B_K(x_0)$ with $K=4$. }
 
 \end{figure}
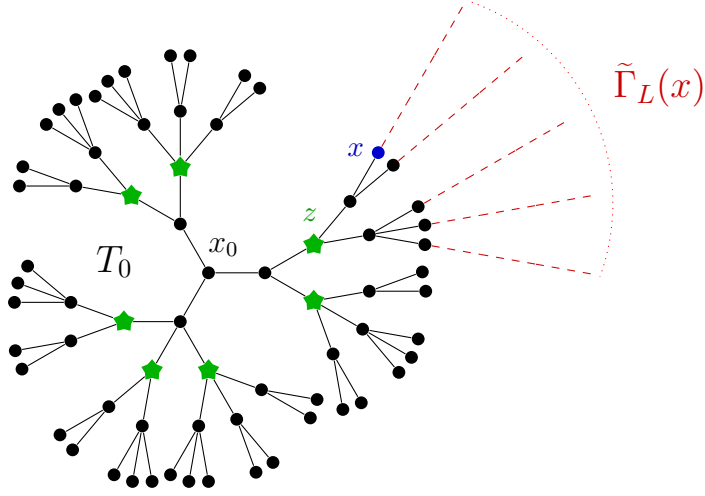
In what follows, when referring to the $\sigma$-fields generated by  $\Gamma'_L(x)$ and $\widetilde\Gamma_L(x)$ we include the information of which  vertices are marked \exit\ (in our setting of $p_1=0$ per~\eqref{eq-Z-min-degree-hypo}, these are but the leaves in these trees).

\begin{lemma}\label{lem:mu-Lambda-exp}
Let $\epsilon>0$, let $t_0$ and $\gamma_1$ as given in~\eqref{eq:t0}, and let $T_0$ be a tree such that $\P(B_K(x_0)=T_0)>0$. For every $x\in\partial T_0$ there exists an event $\cE^G_x\in\sigma\left(\Gamma'_L(x)\right)$ such that 
\begin{equation}
\label{eq:Ax-leaf-x-weaker}
\cE^G_x \subset \cA^G_x(t_0, \gamma_1,\epsilon)\qquad\mbox{ and }\qquad
    \bP^{T_0}\left(\cE^G_x\right) \geq 1-\epsilon\,.
\end{equation}
Moreover, there exists an event $\widetilde\cE^G_x\in\sigma\big(\widetilde\Gamma_L(x)\big)$ such that 
\begin{equation}
\label{eq:Ax-leaf-x}
\widetilde\cE^G_x \subset \cA^G_x(t_0, \gamma_1,2\epsilon)\qquad\mbox{ and }\qquad
    \bP^{T_0}\big(\widetilde\cE^G_x\big) \geq 1-2\epsilon\,.
\end{equation}
\end{lemma}
\begin{proof}
To prove \eqref{eq:Ax-leaf-x-weaker}, we describe in more detail the arguments used in~\cite{BLPS} to derive the aforementioned facts (i)--(iv), and how these are affected by taking $x\in\partial B_K(x_0)$. 
\begin{enumerate}[(i)]
\item The truncated GW-tree from~\cite{BLPS} is a subtree of the standard GW-tree with variable $Z$, where the exploration of each level in the tree may skip certain vertices (where we say a truncation occurs), depending on a certain criterion (see~\cite[Eq.~(3.8)]{BLPS}), which is measurable w.r.t.\ to the currently exposed subtree. No truncation occurs in the first $R$ levels, immediately implying Property~\eqref{it-first-R-min-degree}. 	

  Using $c = c_0 \vee 2\gamma_\star$ for the definition of $R$ in~\eqref{eq:def-L-R}, with $\gamma_\star$ from the definition of~$K$ in~\eqref{eq:t1-K-def},  indeed $B_K(x_0) \subset B_R(x)$ in $G$, which is not truncated by definition.
     
 \item Property~\eqref{it-tree-size} was a consequence of the truncation criterion, which censors the parts of the tree  where the random walk is less likely to visit; as such, the more likely subset of each level of the tree cannot be too large by conservation of mass. 

   Having merely adjusted  $R$ to be at least $2K$, this bound remains valid unchanged.

\item The coupling---now to SRW on a tree rooted at $x$, containing $B_K(x_0)$ with truncated GW-trees at $x$ and all its leaves---is successful unless one of the following occurs:
 \begin{enumerate}[(a)]
 \item  \emph{cycles}: SRW on $G$ encounters a vertex at which exploration has been stopped because of a cycle (thereby hitting a vertex already marked \exit\ in $\Gamma_L(x)$): this 
 is controlled by the size of $\Gamma'_L(x)$, and is thus unchanged (for the new $R$);
  \item \emph{truncation} --- SRW on $G$ visits a truncated vertex (thereby hitting an \exit\ vertex of $\Gamma'_L(x)$): this is controlled by the truncation criterion, thus unchanged (for the new $R$);
\item \emph{degrees}: SRW on $G$ hits a vertex whose degree is inconsistent with its analog in the GW-tree: again, this is controlled by the size of $\Gamma'_L(x)$, thus unchanged.
 \end{enumerate}
 Note that encountering the truncation event (a) along the random walk up to time $t_0$ had probability $O(\epsilon)$, whereas  each of the other two events had probability $o(1)$ (in every given time step the events (b) and (c) had probability $\exp(-c\sqrt{\log n})$, outweighing a union bound over $t_0 = O(\log n)$ steps).

\item \label{it:event-upsilon} The event $\cE^G_x$  was given by $\{ \P^G_x(\Upsilon) \geq 1-\varepsilon\}$ where $\Upsilon$ was the intersection of $\Pi_{t_0}$ (which was denoted $\Upsilon_2$ in~\cite{BLPS}) with two events, each occurring w.h.p.: 
 \begin{itemize}
 \item	$\Upsilon_1$ said that the distance of SRW on the standard GW-tree, at the target time $t_0$, is within a given number of standard deviations from its mean; 
 \item $\Upsilon_3$ said that the (loop-erasure of) SRW on $\Gamma'_L(x)$ at time $t_0$ will belong to the subtree of a given set of vertices (denoted $S''_{\ell_0}$ in that paper), where the probability that SRW visits any given vertex in it is at most $n^{-1}\exp(\gamma_1\sqrt{\log n})$.
 \end{itemize}

The event $\Pi_{t_0}$ was addressed in the previous item. As for the two events $\Upsilon_1, \Upsilon_3$, each occurs w.h.p.\ also in the situation where the truncated GW-tree rooted at $x$ is further attached to a tree given by $B_K(x_0)$ (with additional truncated GW-trees rooted at all other vertices of $\partial B_K(x_0)$), as a consequence of the fact that each vertex of $B_K(x_0)$  has at least 2 offspring. This implies that SRW will encounter a regeneration point below level $K$ by a time period that is $O(K)$ in expectation. In particular, for each of the above events, the estimates of SRW on the (standard) truncated GW-tree apply, once we ignore an additive time shift of, say, $K^2$ steps. 
\end{enumerate}
Altogether, we have verified~\eqref{eq:Ax-leaf-x-weaker}. Now to prove \eqref{eq:Ax-leaf-x}, we have to verify an additional property allowing us to consider an event that is measurable with respect to $\widetilde\Gamma_L(x)$. To this end, let us define
\[
\tilde\tau = \inf\left\{ t>0 \,:\; X_t\in B_{\lfloor K/2\rfloor}(x_0)  \right\}\, ,
\]
and consider the event
\[
\widetilde\cE^G_x = \left\{ \P^G_x\left(\Upsilon\cap\{\tilde\tau> t_0\}\right)>1-2\varepsilon\right\}\, ,
\]
where the event $\Upsilon$ is from item~\eqref{it:event-upsilon}. On the event $\Upsilon\cap\{\tilde\tau> t_0\}$, the walk does not  hit any \exit\ vertex of $\Gamma'_L(x)$ nor does it visit $B_{\lfloor K/2\rfloor}(x_0)$ before time $t_0$. In particular, it does not hit any \exit\ vertex $\widetilde\Gamma_L(x)$ before time $t_0$ and $\widetilde\cE^G_x\in \sigma\left(\widetilde\Gamma_L(x)\right)$. Moreover,
\[
\widetilde\cE^G_x \subset \left\{ \P^G_x\left(\Upsilon\right)>1-2\varepsilon\right\} \subset \cA^G_x(t_0,\gamma_1, 2\varepsilon)\, .
\]
Let us now show that 
\begin{equation}
\label{eq:prob-tildeA}
\bP^{T_0}(\widetilde\cE^G_x)\geq 1-2\varepsilon\, .
\end{equation}
Using two successive union bounds (one on $\P^G_x$, the other on $\bP^{T_0}$), we have
\[
\bP^{T_0}\big(\widetilde\cE^G_x\big) \geq   \bP^{T_0}\left(\P^G_x(\Upsilon)\geq 1-\varepsilon\right)- \bP^{T_0}\left(\P^G_x(\tilde\tau\leq t_0)>\varepsilon\right)\, .
\]
On the one hand, by~\eqref{eq:Ax-leaf-x-weaker}, we know that $\bP^{T_0}\left(\P^G_x(\Upsilon)\geq 1-\varepsilon\right)=\bP^{T_0}\left(\cE^G_x\right)\geq 1-\varepsilon$. On the other hand, by Markov's Inequality,
\[
\bP^{T_0}\left(\P_x^G(\widetilde\tau\leq t_0)>\varepsilon)\right)\leq \frac{\bP^{T_0}(\widetilde\tau\leq t_0)}{\varepsilon}\, \cdot
\]
Since each vertex in the tree $T_0$ has at least 2 offspring, $\rho_t := \dist(X_t,x_0)$ dominates a one-dimensional biased random walk with increment $\xi$ given by $\P(\xi=1)=\frac23$ and $\P(\xi=-1)=\frac13$. It then follows  ({\it e.g.}, via the exponential martingale $2^{-\rho_t}$) that
\[ \sup_{ y\in\partial B_{\lceil 3K/4\rceil }(x_0)}\bP^{T_0}_y\left( \tau_{\lfloor K/2\rfloor} < \tau_{K} \right) \leq 2^{-K/4}\]
for
\[ \tau_\ell = \min\{ t : \rho_t =\ell\}\,,\]
with which we can afford a union bound over at most $C\log n$ time-points serving as the potentially first visit to $\partial B_{\lceil 3K/4\rceil}(x_0)$ before encountering the event $\tau_{\lfloor K/2\rfloor}<\tau_K$, yielding $\bP^{T_0}(\tau_{\lfloor K/2\rfloor}\leq t_0)=o(1)$, and thus, for $n$ large enough,
\[
\bP^{T_0}\left(\P_x^G(\widetilde\tau\leq t_0)>\varepsilon)\right)\leq \varepsilon\, ,
\]
establishing~\eqref{eq:prob-tildeA} and concluding the proof of Lemma~\ref{lem:mu-Lambda-exp}. 
\end{proof}

Our next goal is showing that $\mu_{x_0}(\Lambda_{t_1})$ is concentrated around its mean with an exponential tail. (Recall that $\Lambda_{t_1} \subset \partial B_K(x_0)$ and  $|\partial B_K(x_0)| \geq \log^c n$ for some $c>0$, whereas our target error probability in Theorem~\ref{thm:mu(Bt)} is $O(n^{-2})$.)

To this end, let $\cS=\partial B_{\lfloor K/2\rfloor }(x_0)$ and $z_1,\dots, z_{|\cS|}$ an ordering of the elements of $\cS$ and let
\[
 V_{z_i}= \{x\in \partial T_0,\; \mbox{$x$ is a descendant of $z_i$}\}\, .
 \] 
Sequentially for $i$ going from $1$ to $|\cS|$ and for $x\in V_{z_i}$, we expose the trees $\widetilde\Gamma_L(x)$ (in particular, if upon matching half edge $(u,*)$ during the exposure of $\widetilde\Gamma_L(x)$, we happen to reach a vertex in the boundary of the previously exposed trees, then we stop the exploration at $u$). While doing so, if at some stage $i$, we attempt to match half-edge $(u,*)$ with a half-edge $(v,*)$ where $v\in V_{z_j}$ for some $j>i$, then we mark $v$ as \red\;.

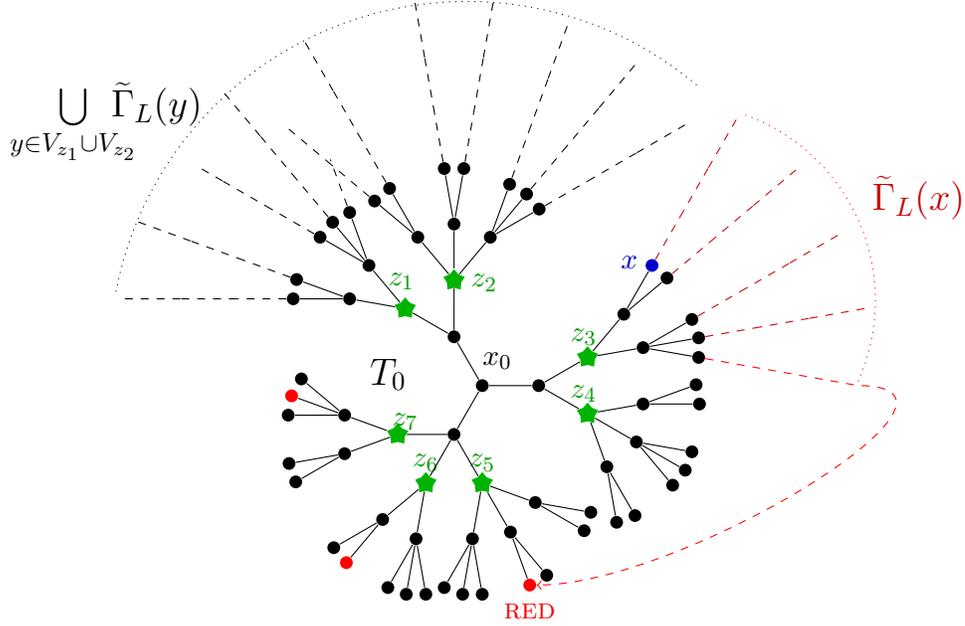
\begin{figure}
\label{fig:exposure2}

\tikzstyle{level 1}=[sibling angle=120]
\tikzstyle{level 2}=[sibling angle=60]
\tikzstyle{level 3}=[sibling angle=40]
\tikzstyle{level 4}=[sibling angle=20]

\begin{tikzpicture}[scale=0.5,sibling distance=8em,grow cyclic,
  every node/.style = {fill,circle,minimum size=1.7mm,inner sep=0}]
  \node[label={[xshift=0.2cm, yshift=0.01cm]
  {$x_0$}}] { }
    child { node { } 
       child { node[green!70!black,label={[xshift=0.1cm, yshift=-0.2cm]
       {\color{green!60!black}$z_7$}},star,minimum size=0.3cm] { } 
           child { node { } 
              child { node { } }
              child { node[red,minimum size=1.7mm] { } }
              child { node { } }
              }
           child { node { } 
              child { node { } }
              child { node { } }
              }
           }        
       child { node[green!70!black,label={[xshift=0, yshift=-0.1cm]
       {\color{green!60!black}$z_6$}},star,minimum size=0.3cm] { } 
           child { node { } 
              child { node { } }
              child { node[red,minimum size=1.7mm] { } }
              }
           child { node { } 
              child { node { } }
              child { node { } }
              child { node { } }
              }
           }
       child { node[green!70!black,label={[xshift=0, yshift=-0.1cm]
       {\color{green!60!black}$z_5$}},star,minimum size=0.3cm] { } 
           child { node { } 
              child { node { } }
              child { node { } }
              child { node { } }
              }
           child { node { } 
              child { node[red,label={[yshift=-0.8cm]$\color{red}\textsc{red}$},minimum size=1.7mm] (B) { } }
              child { node { } }
              }
           child { node { } 
              child { node { } }
              child { node { } }
              }
           }
        }
    child { node { } 
       child { node[green!70!black,label={[xshift=-0.05cm, yshift=-0.1cm]
       {\color{green!60!black}$z_4$}},star,minimum size=0.3cm] { } 
           child { node { } 
              child { node { } }
              child { node { } }
              }
           child { node { } 
              child { node { } }
              child { node { } }
              child { node { } }
              }
           child { node { } 
              child { node { } }
              child { node { } }
              }
           }
       child { node[green!70!black,label={[xshift=-0.05cm, yshift=-0.1cm]
       {\color{green!60!black}$z_3$}},star,minimum size=0.3cm] { } 
           child { node { } 
              child { node (A) { }  }
              child { node { } child[red!80!black,dashed]{child[red!80!black,dashed]{child[red!80!black,dashed]{}}} }
              child { node { } child[red!80!black,dashed]{child[red!80!black,dashed]{child[red!80!black,dashed]{}}} }
              }
           child { node { } 
              child { node { } child[red!80!black,dashed]{child[red!80!black,dashed]{child[red!80!black,dashed]{}}} }
              child { node[blue!80!black,label={[xshift=-0.3cm,yshift=-0.2cm]
              {\color{blue!80!black}$x$}}] { } child[red!80!black,dashed]{child[red!80!black,dashed]{child[red!80!black,dashed]{}}} }
              }
           }
       }
    child { node { } 
       child { node[green!70!black,label={[xshift=0.4cm, yshift=-0.4cm]
       {\color{green!60!black}$z_2$}},star,minimum size=0.3cm] { } 
           child { node { } 
              child { node { } child[black,dashed]{child[black,dashed]{child[black,dashed]{}}}}
              child { node { } child[black,dashed]{child[black,dashed]{child[black,dashed]{}}}}
              child { node { } child[black,dashed]{child[black,dashed]{child[black,dashed]{}}}}
              }
           child { node { } 
              child { node { } child[black,dashed]{child[black,dashed]{child[black,dashed]{}}} }
              child { node { } child[black,dashed]{child[black,dashed]{child[black,dashed]{}}}}
              }
           child { node { } 
              child { node { } child[black,dashed]{child[black,dashed]{child[black,dashed]{}}}}
              child { node { } child[black,dashed]{child[black,dashed]{}} }
              }
           }
       child { node[green!70!black,label={[xshift=-0.05cm, yshift=-0.3]
       {\color{green!60!black}$z_1$}},star,minimum size=0.3cm] { } 
           child { node { } 
              child { node { } child[black,dashed]{} }
              child { node { } child[black,dashed]{child[black,dashed]{child[black,dashed]{}}} }
              child { node { }child[black,dashed]{child[black,dashed]{child[black,dashed]{}}} }
              }
           child { node { } 
              child { node { } child[black,dashed]{child[black,dashed]{child[black,dashed]{}}}}
              child { node { } child[black,dashed]{child[black,dashed]{child[black,dashed]{}}} }
              }
           }
       };
  \node[white] at (11.6,5) {\Large\color{red!80!black}$\widetilde\Gamma_L(x)$}; 
    \node[white] at (-2.5,0.3) {\Large\color{black} $T_0$};
    \node[white] at (-10,7) {\Large\color{black} $\displaystyle{\bigcup_{y\in V_{z_1}\cup V_{z_2}}\!\!\!\!\!\widetilde\Gamma_L(y)}$}; 
    \draw [->,red,dashed] (A) to (10,0) to [out=0,in=0] (B);
    \draw[dotted,red] (7.,7.2) arc (70:-25:5.25);
    \draw[dotted,black] (5.7,8.) arc (50:170:9.5);  
\end{tikzpicture}

\caption{The $i^{\mbox{th}}$ stage of the sequential exposure process ($i=3$) with the \red\ coloring of vertices.}
 
 \end{figure}

\begin{lemma}
\label{lem:number-reds}
Let $\bR_L$ be the number of \red\ vertices in $\partial T_0$ at the end of the exploration process. Then, for $\gamma_1$ as given in~\eqref{eq:t0} and every sufficiently large $n$,
\[
\bP^{T_0}\left(\bR_L >15\gamma_1^{-1}\sqrt{\log n}\right) < n^{-2}\, .
\]
\end{lemma}
\begin{proof}
By our assumption on the maximum degree, at any stage of the exploration process, the number of unmatched half-edges attached to a vertex of $\partial T_0$ is smaller than $\Delta^{K+1}\leq \exp\left((\log n)^{\frac{1-\delta}{2}}\right)$ for large enough $n$. Now recall that the truncation ensures that each of the trees $\widetilde\Gamma_L(x)$ has size at most $n\exp\left(-\frac{1}{4}\gamma_1\sqrt{\log n}\right)$, implying that the total number of pairs formed during the exploration process is at most $n\exp\left(-\frac{1}{5}\gamma_1\sqrt{\log n}\right)$. All in all, the probability that at least $C\sqrt{\log n}$ times an unmatched half-edge attached to a vertex of $\partial T_0$ is chosen as target during the process is smaller than
\[
{n\exp\left(-\frac{1}{5}\gamma_1\sqrt{\log n}\right) \choose C\sqrt{\log n}} \left(\frac{\exp\left[(\log n)^{\frac{1-\delta}{2}}\right]}{(3-o(1))n}\right)^{C\sqrt{\log n}}  \, .
\]
Using that ${a\choose b}\leq a^b$ for all $a,b\geq 0$, we see that
\[
\bP^{T_0}\left(\bR_L >C\sqrt{\log n}\right)\leq \exp\left(-\frac{C\gamma_1}{5}\log n +C(\log n)^{1-\frac{\delta}{2}}\right)\, ,
\]
which is smaller than $n^{-2}$ for $C\geq15/\gamma_1$ and large enough $n$.
\end{proof}

Let us denote by $\cF_0$ the $\sigma$-field generated by $T_0$, and for $i=1,\dots, |\cS|$, let $\cF_i$ be the $\sigma$-field generated by $T_0$ and $\bigcup_{j=1}^i\bigcup_{x\in V_{z_j}}\widetilde\Gamma_L(x)$, together with the \red\ coloring of vertices of $\partial T_0$ up to stage  $i$. We say that a vertex $z\in \cS$ is \nice\ if none of its descendants in $\partial T_0$ is colored \red. Note that the event $\{\mbox{$z_i$ is \nice}\}$ belongs to $\cF_{i-1}$.

\begin{lemma}
\label{lem:exp-cond}
For $i\geq 1$, consider running the exploration process up to stage $i-1$, and let $x\in\partial T_0$ be a descendant of $z_i$. Then
\[
\P\left({\widetilde\cE}^G_x \mid\ \cF_{i-1}\right)\geq \ind_{\{\mbox{$z_i$ is \nice}\}}(1-2\varepsilon)\, .
\]
\end{lemma}
\begin{proof}
The proof of Lemma~\ref{lem:exp-cond} closely parallels the proof of Lemma~\ref{lem:mu-Lambda-exp}. Indeed, all that remains to show is that having exposed the trees $\widetilde\Gamma_L(y)$ for $y\in V_{z_1}\cup\dots\cup V_{z_{i-1}}$ does not essentially change the probability of $\widetilde\cE_x$ for $x\in V_{z_i}$, at least when $z_i$ is \nice. First note that $\widetilde\cE_x\in\sigma\left(\widetilde\Gamma_L(x)\right)$ and that the tree $\widetilde\Gamma_L(x)$ does not intersect any of the trees $\widetilde\Gamma_L(y)$ for $y\in V_{z_1}\cup\dots\cup V_{z_{i-1}}$ (except for at $T_0$). Now, the only difference with the setting of Lemma~\ref{lem:mu-Lambda-exp} is that the number of exposed half-edges may now be larger. However, by the truncation criterion, we know that for all $y\in V_{z_1}\cup\dots\cup V_{z_{i-1}}$, the size of $\widetilde\Gamma_L(y)$ is smaller than $n\exp\left(-\frac{1}{4}\gamma_1\sqrt{\log n}\right)$, so that the total number of exposed half-edges at the end of stage $i-1$ is at most $\Delta^K n\exp\left(-\frac{1}{4}\gamma_1\sqrt{\log n}\right)\leq n\exp\left(-\frac{1}{5}\gamma_1\sqrt{\log n}\right)$. Recalling that the the three possible impediments to a successful coupling with a truncated GW-tree (\emph{truncation}, \emph{cycles} and \emph{degrees}) were controlled either by the truncation criterion itself (thus unchanged), or by the upper bound on the size of the exposed subgraph (which has merely increased from $n\exp\left(-\frac{1}{4}\gamma_1\sqrt{\log n}\right)$ to $n\exp\left(-\frac{1}{5}\gamma_1\sqrt{\log n}\right)$) and noticing that the event of a visit to $\partial B_{\lfloor K/2\rfloor }(x_0)$ before time $t_0$ is not affected by this extra-exposure, this is enough to ensures that, conditionally on $\cF_{i-1}$, for all $x\in V_{z_i}$ with $z_i$ \nice\,, the coupling of SRW on $G$ started at $x$ with SRW on a tree rooted at $x$ (containing $B_K(x_0)$ with truncated GW-trees on those leaves which are descendants of $z_i$) is successful with large probability. Also, the events $\Upsilon_1$ and $\Upsilon_3$ can be handled exactly as in the proof of Lemma~\ref{lem:mu-Lambda-exp}, a regeneration point below level $K$ being quickly found by the walk. 
\end{proof}

\begin{lemma}
Let $\epsilon>0$, take $t_0$ and $\gamma_1$ as in~\eqref{eq:t0}, and let $\mu_{x_0}=\P_{x_0}^G(X_{\tau_K}\in\cdot)$ and 
\begin{equation}\label{eq:U-def} 
U = \left\{ x \,:\; \widetilde\cE^G_x \mbox{ does not hold\,}\right\}\,.	
\end{equation}
If $T_0$ is a tree with $\P(B_K(x_0)=T_0)>0$ 
then for every sufficiently large $n$,
\[\bP^{T_0}\left(\mu_{x_0}(U)  \geq 4\epsilon \right) < 2 n^{-2}\,.
\]
\end{lemma}
\begin{proof}
Note that
\begin{equation}
\label{eq:nice_not-nice}
\mu_{x_0}(U)\leq \sum_{i=1}^{|\cS|}\mu_{x_0}(V_{z_i})\ind_{\{\mbox{$z_i$ is not \nice}\}} + \sum_{i=1}^{|\cS|} \mu_{x_0}(V_{z_i}\cap U)\ind_{\{\mbox{$z_i$ is \nice}\}}\, .
\end{equation}
As for the first term in the right-hand side of \eqref{eq:nice_not-nice}, we have
\[
\sum_{i=1}^{|\cS|}\mu_{x_0}(V_{z_i})\ind_{\{\mbox{$z_i$ is not \nice}\}}\leq \max_{z\in \cS} \mu_{x_0}(V_{z})\bR_L\, .
\]
Observe that
\begin{equation}
\label{eq:max-mu-V_i}
\max_{z\in \cS} \mu_{x_0}(V_z) \leq \bP_{x_0}^{T_0} (\tau_z < \tau_{\partial T_0}) \leq 3\cdot 2^{-\lfloor K/2\rfloor}\,,
\end{equation}
as the probability of ever visiting a vertex at level $k$ in the infinite binary tree is at most $3 \cdot 2^{-k}$ (at each point along the path from the root to this vertex, the random walk has a probability of $1/3$ of escaping to infinity through an alternative branch), and adding edges can only decrease the probability of visiting a vertex. And by Lemma~\ref{lem:number-reds}, the number of \red\ vertices is smaller than $\frac{15}{\gamma_1}\sqrt{\log n}$ with probability at least $1-n^{-2}$. Choosing $\gamma_\star$ large enough in the definition of $K$ then ensures that
\[
 \max_{z\in \cS} \mu_{x_0}(V_{z})\bR_L \leq \varepsilon\, ,
 \]
 with probability at least $1-n^{-2}$.
 
 Moving on to the second term in the right-hand side of \eqref{eq:nice_not-nice}, let
\[W = \sum_{i=1}^{|\cS|} W_i\quad\mbox{ for }\quad W_i = \mu_{x_0}(V_{z_i}\cap U)\ind_{\{\mbox{$z_i$ is \nice}\}}\,,\]
and
\[M_t := \sum_{i\leq t}(W_i - \E[W_i\mid\cF_{i-1}])\qquad (t=1,\ldots,|\cS|)\,.\]
Note that, for all $1\leq i\leq |\cS|$, the variable $\sum_{j<i} W_j$  is $\cF_{i-1}$-measurable and 
\[
\E\left[ W_i\mid \cF_{i-1}\right]=\ind_{\{\mbox{$z_i$ is \nice}\}} \sum_{x\in V_{z_i}}\mu_{x_0}(x)\P\left( \big(\widetilde\cE^G_x\big)^{c} \mid \cF_{i-1}\right) \leq  2\varepsilon \mu_{x_0}(V_{z_i})\, ,
\]
with the last inequality  by Lemma~\ref{lem:exp-cond}.
In particular,    $W \leq M_{|\cS|} + 2\epsilon$. Also, by \eqref{eq:max-mu-V_i},
\[ \sum_{z\in\cS} \mu_{x_0}(V_{z})^2 \leq \max_{z\in\cS} \mu_{x_0}(V_z) \leq 3\cdot 2^{-\lfloor K/2\rfloor}\,,\]
and so $\sum_t\big|M_t - M_{t-1}\big|^2 \leq 3\cdot 2^{-\lfloor K/2\rfloor}$ with probability 1. 
Thus, we can infer from the Hoeffding--Azuma inequality for the martingale $(M_t)$ that
\[ \bP^{T_0}\left( W\geq  3\epsilon \right) \leq
 \bP^{T_0}\left( M_{|\cS|}\geq  \epsilon \right) \leq \exp\left(-\frac{ \epsilon^2}{6\cdot 2^{-\lfloor K/2\rfloor}}\right) < n^{-2}\,,\]
provided that $\gamma_\star$ in the definition~\eqref{eq:t1-K-def} is chosen to be sufficiently large.	
\end{proof}

Together with~\eqref{eq:gap-boost} this completes the proofs of Theorem~\ref{thm:mu(Bt)} and Proposition~\ref{prop:worst-case}.
\end{proof}

\section{Entropy comparison of walks on Galton--Watson trees}\label{sec:entropy}
In this section we prove Proposition~\ref{prop:nbrw-faster}, showing that the ratio $\bh_X / \bh_Y$ (which, as established in Proposition~\ref{prop:worst-case}, is the ratio between the cutoff locations for SRW and NBRW on our sparse random graphs) is at most some $c(Z) < 1$. Assume w.l.o.g.\ that $Z$ is non-constant ({\it i.e.}, $\P(Z \neq \E Z) > 0$), otherwise this ratio is $\frac{\E Z-1}{\E Z+1}$ as mentioned above.

Let $(T,\rho)$ be a rooted {\it Augmented} Galton--Watson tree
({\it i.e.}, the tree formed by joining the roots (one of which being $\rho$) of two i.i.d.\ Galton--Watson trees by an edge) with offspring variable $Z$. As before,
let $(X_t)$ and $(Y_t)$ be SRW and NBRW on $T$, resp.; as first observed in~\cite{lyons1995ergodic},  ($T,\rho,{\rm SRW}$) is a \emph{stationary environment}, {\it i.e.}, $(T,\rho) \stackrel{{\rm d}}=(T,\P_\rho(X_1\in\cdot))$.
Conditioned on $(T,\rho)$, let $H_t(T,\rho)$ be the entropy of SRW after $t$ steps:
\[
H_t(T,\rho)=H\Big(\P_\rho(X_t\in\cdot \mid T)\Big)\quad\mbox{ and }\quad h_t=\E[H_t(T,\rho)]\,,
\]
and similarly defined $L_t(T,\rho)$ for the NBRW by
\[
L_t(T,\rho)=H\Big(\P_\rho(Y_t\in\cdot \mid T)\Big)\quad\mbox{ and }\quad \ell_t=\E[L_t(T,\rho)]\,.
\]
With these notations, we have
\begin{equation}\label{eq:entropySRW_NBRW}
\bh_X=\lim_{t\to\infty}\frac{h_t}{t}\,,\quad\mbox{ and }
\bh_Y=\lim_{t\to\infty}\frac{\ell_t}{t}\,,
\end{equation}
where the identity for $\bh_X$ is by~\cite[Theorem 9.7]{lyons1995ergodic}, and $\ell_t$ ($t \geq 1$) is explicitly given by
\[ \ell_t=\E[\log(Z+1)] + (t-1) \E[\log Z] = \E[\log(Z+1)] + (t-1)\bh_Y\,.\]
Since $X_1$ and $Y_1$ have the same distribution, we further have
\[ h_1 = \ell_1 = \E[\log(Z+1)]\,.\]
We need the following result (cf., e.g., the proof of Theorem 3.2 in~\cite{benjamini2012ergodic} and Corollary~10 in~\cite{benjamini2015}), which was first observed in the case of random walks on groups by~\cite{kaimanovich1983random}. (Entropy of random walks on random stationary environments were thereafter studied in~\cite{kaimanovich1990boundary}). 
We include the short proof for completeness.
\begin{claim}\label{clm:ht-nonincrease}
The map $t\mapsto (h_t -h_{t-1})$ is non-increasing.
\end{claim}
\begin{proof}Consider the joint entropy of $X_1$ and $X_t$ given $T$:
\[
H_{1,t}(T,\rho):=H\Big(\Psubbig_\rho{(X_1,X_t)\in\cdot \mid T}\Big)\quad\mbox{ and }\quad h_{1,t}:=\E[H_{1,t}(T,\rho)]\, .\]
Factoring out $ \P_\rho(X_1=x \mid T) $ from $\P_\rho(X_1=x,X_t=y\mid T)$, one sees that
\[
H_{1,t}(T,\rho)=H_1(T,\rho)-\sum_{x\in T} \P_\rho(X_1=x\mid T)\sum_{y\in T}\P_x(X_{t-1}=y\mid T)\log \P_x(X_{t-1}=y \mid T)\, ,
\]
and taking expectation gives
\[
h_{1,t}=h_1+\E[H_{t-1}(T,\P_\rho(X_1\in\cdot))]=h_1+h_{t-1}\, ,
\]
where the last equality is due to the stationarity of the environment
. Therefore,
\begin{displaymath}
h_t-h_{t-1}=h_t-h_{1,t}+h_1=\E[H_t(T,\rho)-H_{1,t}(T,\rho)]+h_1\, .
\end{displaymath}
Conditioned on $T$, the term $H_{1,t}(T,\rho)-H_t(T,\rho)$ is the conditional entropy $H(X_1\mid X_t$), which satisfies
$
H(X_1\mid X_t)=H(X_1\mid X_t,X_{t+1})\leq H(X_1\mid X_{t+1})
$,
since $X_1,X_{t+1}$ are conditionally independent given $X_t$, and extra information cannot increase entropy. So,
\begin{equation*}
h_t-h_{t-1}\geq  h_{t+1}-h_{1,t+1}+h_1=h_{t+1}-h_t\, .\qedhere
\end{equation*}
\end{proof}
The fact that $(h_t - h_{t-1})$ is non-increasing in $t$ implies that, for every $t$,
\[ h_t - h_2\leq \lceil \tfrac{t-2}2\rceil (h_3-h_1) \,,\]
from which we see (recalling~\eqref{eq:entropySRW_NBRW}) that it suffices to show that $h_3 - h_1 <  2 \bh_Y$ in order to conclude that $\bh_X = \lim_{t\to\infty} h_t/t < \bh_Y$. (Note that, while establishing the inequality $h_2 - h_1 <  \bh_Y$ would also suffice---and indeed that holds provided that $Z\geq 3$---it fails in general for $Z\geq 2$; {\it e.g.}, for $Z=2$ one has $h_2-h_1=\frac23\log 3 > \log 2 = \bh_Y$, and perturbing $Z$ to be $3$ with a suitably small probability leads to a similar behavior.) Towards establishing this, consider $H_3(T,\rho)$, the entropy of SRW after 3 steps on the tree $T$; by denoting  $T_k = \{ z\in T : \dist(\rho,z) = k\}$, one has that $H_3(T,\rho)=R+S$ where
\begin{align}
R&=-\sum_{z\in T_3}\P_\rho(X_3=z)\log\P_\rho(X_3=z)\, ,\label{eq-H3-R}\\
S&=-\sum_{x\in T_1}\P_\rho(X_3=x)\log \P_\rho(X_3=x)\, .\nonumber
\end{align}
Using the notation $y\prec x$ to denote that $y$ is a child of $x$ and $D_y$ to be the number of children of $y$ in $T$, one has
\begin{align*}
\E[ R \mid T_1, T_2, T_3 ] = \sum_{y\prec x\prec \rho}  D_y \frac{\log D_\rho +\log (D_x+1)+\log (D_y+1)}{D_\rho (D_x+1) (D_y+1)}
\end{align*}
(where each child of $y$ played the role of $z$ in~\eqref{eq-H3-R}, hence the factor $D_y$ above), thus
\begin{align*}
\E[ R \mid T_1, T_2 ] &= \frac{1}{D_\rho} \sum_{y\prec x\prec \rho } \frac1{D_x+1} \left(\E\Big[\frac{Z}{Z+1}\Big](\log D_\rho + \log (D_x+1)) + \E\Big[\frac{Z\log (Z+1)}{Z+1}\Big] \right)\\
&= \frac{1}{D_\rho} \sum_{x\prec \rho} \frac{D_x}{D_x+1} \left( \E\Big[\frac{Z}{Z+1}\Big]\left ( \log D_\rho+\log (D_x+1)\right) + \E\Big[\frac{Z \log (Z+1)}{Z+1}\Big]\right)\,.
\end{align*}
Continuing in the same manner, setting $\beta := \E\left[\frac{Z}{Z+1}\right]$ to simplify the notation, we have
\begin{align}
\E[R] &= \E\Big[ \E[ R \mid T_1] \Big]= \beta^2 \E\left[ \log D_\rho\right] +2 \beta \E\left[\frac{Z\log (Z+1)}{Z+1}\right]\nonumber\\
&=\beta^2 \E\left[ \log (Z+1)\right]  +2 \beta \E\Big[\frac{Z\log (Z+1)}{Z+1}\Big]\label{eq-H3-R-bound}
\end{align}
where the factor $1/D_\rho$ canceled the number of i.i.d.\ choices corresponding to $x\prec\rho$, and of the three summands corresponding to $x\prec \rho$ in the preceding display, each of the two that did not involve $D_\rho$ yielded the same expression after averaging over $T_2$.

Turning our attention to $S$, by the convexity of $x\mapsto x\log x$ and Jensen's inequality for conditional expectation,
\begin{align}\label{eq-H3-S-bound}\E[S\mid T_1 ]\leq -\sum_{x\in T_1}\E\big[\P_\rho(X_3=x)\mid T_1 \big]\log \E\big[\P_\rho(X_3=x)\mid T_1\big]\, .
\end{align}
For every $x\in T_1$, accounting for whether $X_2=\rho$ or $X_2=y$ for some $y\prec x$ shows that
\[\P_\rho(X_3=x)= \frac{1}{D_\rho}\bigg( \sum_{x'\in T_1}\frac{1}{D_\rho (D_{x'}+1)}+\sum_{y\prec x}\frac{1}{(D_x +1)(D_y+1)}\bigg)\,;
\]
thus, again using that $\E[ \sum_{y\prec x} \frac{1}{(D_x +1)(D_y+1)} \mid T_1,T_2] = \E[\frac{1}{Z+1}]\frac{D_x}{D_x+1}$ as reasoned above,
\[
\E\left[\P_\rho(X_3=x)\mid T_1 \right]=\frac{\E\Big[\frac{1}{Z+1}\Big]+ \E\Big[\frac{1}{Z+1}\Big]\E\Big[\frac{Z}{Z+1}\Big]}{D_\rho}=
\frac{1-\beta+(1-\beta)\beta}{D_\rho} = \frac{1-\beta^2}{D_\rho} \, .
\]
Plugging this in~\eqref{eq-H3-S-bound} shows that
\[ \E[S \mid T_1] \leq (1-\beta^2)\log\frac{D_\rho}{1-\beta^2}\,,\]
which, after taking an average over $T_1$ and combining it with~\eqref{eq-H3-R-bound}, yields
\begin{align}\label{eq-H3-R-S-bound}
h_3 = \E[R + S] \leq \E[\log(Z+1)] + 2\beta\E\left[\frac{Z\log(Z+1)}{Z+1}\right] - (1-\beta^2)\log(1-\beta^2)\,.
\end{align}
By Jensen's inequality and the fact that $Z$ is non-constant (thus the same holds also for $\frac{Z}{Z+1}$), we have  $\beta^2 < \E[(\frac{Z}{Z+1})^2]$,  hence
$\log(1-\beta^2) > \log(\E[1-(\frac{Z}{Z+1})^2])$,
and another application of Jensen's inequality to $x\mapsto \log x$ implies that
\begin{align*} (1-\beta^2)\log(1-\beta^2) &> (1-\beta)\E\left[\frac{2Z+1}{Z+1}\right] \E\left[ \log \left(\frac{2Z+1}{(Z+1)^2} \right)\right] \\
&\geq (1-\beta) \E\left[\frac{2Z+1}{Z+1} \log \left(\frac{2Z+1}{(Z+1)^2} \right)\right]
 \,,
\end{align*}
where the second inequality used that $f_1(z)=\frac{2z+1}{z+1} = 2-\frac1{z+1}$ is increasing whereas $f_2(z)=\log\left(\frac{2z+1}{(z+1)^2}\right) = \log(1-(\frac{z}{z+1})^2)$ is decreasing, implying---noting $\E[f_1(Z)^2]<\infty$ and $\E[f_2(Z)^2]<\infty$ since $Z \geq 1$---that $\cov(f_1(Z),f_2(Z)) \leq 0$.
Revisiting~\eqref{eq-H3-R-S-bound}, while recalling that $h_1 = \E[\log(Z+1)]$, we now infer that
\begin{align}
&h_3 - h_1 < 2 \beta \E\left[\frac{Z\log(Z+1)}{Z+1}\right] - (1-\beta)\E\left[\frac{2Z+1}{Z+1}\log\left(\frac{2Z+1}{(Z+1)^2}\right)\right] \nonumber\\
&= 2 \E\left[\frac{Z\log(Z+1)}{Z+1}\right] - (1-\beta)\E\left[\frac{2Z+1}{Z+1}\log\left(\frac{2Z+1}{(Z+1)^2}\right) + \frac{2Z\log(Z+1)}{Z+1}\right] \nonumber\\
&= 2\E\left[\frac{Z\log(Z+1)}{Z+1}\right] - (1-\beta)\E\left[f_3(Z)\right] \, ,
\label{eq-h3-h1}
\end{align}
where
\[ f_3(z) := \frac{2z+1}{z+1}\log(2z+1) - 2 \log(z+1)\,.\]
It is easy to verify that $f'_3(z) = (z+1)^{-2}\log(2z+1) > 0$ for every $z>0$, thus $f_3(z)$ is increasing and $\E[f_3(Z)^2]<\infty$ thanks to the facts $Z\geq 1$ and $\E Z < \infty$. Therefore, when considered with the function $f_4(z) = \frac{1}{z+1}$ which is decreasing and has a finite second moment, we have $\cov(f_3(Z),f_4(Z)) \leq 0$, or equivalently,
\begin{align*}
(1-\beta) \E\left[f_3(Z)\right] &\geq \E\left[\frac{f_3(Z)}{Z+1}\right] = \E\left[\frac{2Z+1}{(Z+1)^2}\log(2Z+1)-   \frac{2}{Z+1}\log(Z+1)\right] \, .
\end{align*}
Combining these with~\eqref{eq-h3-h1} yields
\[ h_3 - h_1 < 2\E\left[ \log(Z+1) -\frac{2Z+1}{2(Z+1)^2} \log(2Z+1)\right] \,.\]
The proof will thus be concluded by showing that the function
\[  g(x) := \log\left(\frac{x+1}x\right)- \frac{2x+1}{2(x+1)^2}\log(2x+1)\,.\]
satisfies $g(x) < 0$ for all $x\geq 2$, which would then imply that
\[ h_3-h_1 < 2\E[ \log Z] = 2\bh_Y = \ell_3-\ell_1\,.\]
To see this, first observe that
\[ g'(x) = (2x+1)(x+1)-\frac{x^2\log(2x+1)}{x(x+1)^3}\,.\]
Now, along the interval $[1,\infty)$, the function $x\mapsto\log(2x+1)$ increases from $0$ to $\infty$, whereas $x\mapsto 2+3x^{-1} + x^{-2}$ decreases from $6$ to $2$, thus $g'(x)$ has a unique root $x_0$ and is negative on $[1,x_0)$ and positive on $(x_0,\infty)$. Hence, it suffices to show that $g(x)$ is negative at $x=2$ and near $\infty$, which is indeed the case: $g(2)= \log(3/2)-\frac5{18}\log 5 < 0$ and $g(x) x/\log(x) \to -1$ as $x\to\infty$, as claimed.
\qed

\vspace{-5pt}
\subsection*{Acknowledgements} E.L. was supported in part by NSF grant DMS-1513403.
\vspace{-5.pt}
\bibliographystyle{abbrv}

\end{document}